\documentclass[a4paper,10pt]{amsart}
\usepackage[T1]{fontenc}
\usepackage{lmodern}
\usepackage[all]{xy}
\usepackage{amssymb, mathrsfs}
\usepackage{enumitem,nicefrac,mathtools}
\usepackage[british]{babel}
\usepackage{xcolor}

\newcommand*{\MRref}[2]{\linebreak[0] \href{http://www.ams.org/mathscinet-getitem?mr=#1}{MR \textbf{#1}}}

\usepackage[pdftitle={Classification of certain continuous fields},
  pdfauthor={Rasmus Bentmann},
  pdfsubject={Mathematics}
]{hyperref}
\usepackage[lite,non-sorted-cites]{amsrefs}
\usepackage{microtype}

\usepackage{graphicx}
\usepackage{tikz}
\usetikzlibrary{matrix,arrows}

\usepackage{array}
\usepackage[labelformat=simple]{subfig}

\DeclareMathOperator{\Hom}{Hom}

\DeclareMathOperator{\Ext}{Ext}

\newcommand*{\KK}{\textup{KK}}
\newcommand*{\K}{\textup{K}}
\newcommand*{\E}{\textup{E}}

\newcommand*{\XnK}{\mathrm{X}_n\mathrm{K}}
\newcommand*{\OK}{\mathbb{O}\textup{K}}
\newcommand*{\Res}{\textup{Res}}
\newcommand*{\Colim}{\textup{Colim}}

\newcommand*{\into}{\rightarrowtail}
\newcommand*{\prto}{\twoheadrightarrow}

\newcommand*{\KKcat}{\mathfrak{KK}}
\newcommand*{\Ecat}{\mathfrak{E}}

\newcommand*{\C}{\mathbb{C}}

\newcommand*{\Z}{\mathbb{Z}}
\newcommand*{\N}{\mathbb{N}}
\newcommand*{\Open}{\mathbb{O}}     
\newcommand*{\Compacts}{\mathbb{K}}
\newcommand*{\Field}{\mathbb{F}}

\newcommand*{\Q}{\mathbb{Q}}
\newcommand*{\UHF}{M_\Q}

\newcommand*{\nb}{\nobreakdash}  

\newcommand*{\Bootstrap}{{\mathcal B}}
\newcommand*{\Cuntz}{{\mathcal O}}

\newcommand*{\Cstar}{\texorpdfstring{$C^*$\nb-}{C*}}
\newcommand*{\Star}{\texorpdfstring{$^*$\nb-}{*-}}

\newcommand*{\blank}{\text{\textvisiblespace}}

\theoremstyle{plain}

\numberwithin{equation}{section}

\theoremstyle{plain}
\newtheorem{theorem}[equation]{Theorem}

\newtheorem{proposition}[equation]{Proposition}

\theoremstyle{definition}
\newtheorem{definition}[equation]{Definition}

\newtheorem{example}[equation]{Example}

\theoremstyle{remark}
\newtheorem{remark}[equation]{Remark}

\title[Classification of certain continuous fields]{Classification of certain continuous fields of Kirchberg algebras}

\author{Rasmus Bentmann}
\address{Department of Mathematical Sciences\\University of
Copenhagen\\Universitets\-parken~5\\2100 Copenhagen~\O \\Denmark}
\email{bentmann@math.ku.dk}



\thanks{The author was supported by the Danish National Research Foundation through the Centre for Symmetry and Deformation (DNRF92) and by the Marie Curie Research Training Network EU-NCG}

\begin{document}

\begin{abstract}
We show that the $\K$\nb-theory cosheaf is a complete invariant for separable continuous fields with vanishing boundary maps over a finite-dimensional compact metrizable topological space whose fibers are stable Kirchberg algebras with rational $\K$-theory groups satisfying the universal coefficient theorem. We provide a range result for fields in this class with finite-dimensional $\K$\nb-theory. There are versions of both results for unital continuous fields.
\end{abstract}

\maketitle

\section{Introduction}

The present article is part of a programme aimed at deciding when two \Cstar{}al\-ge\-bras over a (second countable) topological space~$X$ are equivalent in ideal-related $\KK$\nb-the\-ory. In consequence of a fundamental result due to Eberhard Kirchberg~\cite{Kirchberg}*{Folgerung~4.3}, this is a central question in the classification theory of non-simple purely infinite \Cstar{}algebras.

We briefly review the existing results in the literature. These are concerned with two classes of spaces, namely finite (non-Hausdorff) spaces on the one hand and finite-dimensional compact metrizable spaces on the other hand. Universal coefficient theorems (UCT), which compute the $\KK(X)$\nb-groups in terms of $\K$\nb-theoretic invariants and which imply a solution to the given classification problem, have been established for certain classes of finite $T_0$-spaces in~\cites{RS,Bonkat:Thesis,Restorff:Thesis,MN:Filtrated,Bentmann:Thesis,BK}. A solution for finite unique path spaces using a more complicated kind of invariant is provided in~\cite{Bentmann-Meyer}. For arbitrary finite spaces, the problem remains unsolved and seems rather unfeasible because certain wildness phenomena occur; see~\cite{Bentmann:Algebraic_models}. In the context of finite-dimensional compact metrizable spaces the strongest results are available in the totally disconnected case~\cites{Dadarlat-Pasnicu:Continuous_fields,DM} and in the case of the unit interval~\cites{DE,Dadarlat--Elliott--Niu,BD}.

As these examples illustrate, the feasibility of the classification problem under consideration depends critically on the space~$X$. However, it is possible to obtain solutions for more general base spaces under additional $\K$\nb-the\-o\-ret\-i\-cal assumptions. For instance, Kirchberg's isomorphism theorem~\cite{Kirchberg}*{Folgerung~2.18} states that two separable nuclear stable $\Cuntz_2$\nb-absorbing \Cstar{}algebras are isomorphic once their primitive ideal spaces are homeomorphic ($\Cuntz_2$\nb-ab\-sorp\-tion entails in particular the vanishing of all $\K$\nb-theoretic data). In~\cite{Bentmann:Real_rank_zero_and_intermediate_cancellation}, a UCT for \Cstar{}algebras with \emph{vanishing boundary maps} (as we shall define in~\S\ref{sec:vanishing_boundary_maps}) over an arbitrary finite $T_0$-space is proven.

The main result of the present article is the following; it is based on the UCT in~\cite{Bentmann:Real_rank_zero_and_intermediate_cancellation} and Dadarlat--Meyer's approximation of ideal-related $\E$\nb-theory over an infinite space~$X$ by ideal-related $\E$\nb-theory over finite quotient spaces of~$X$ from~\cite{DM}, together with Kirchberg's theorem~\cite{Kirchberg}.

\begin{theorem}
  \label{thm:main}
Let $A$ and $B$ be separable continuous fields over a finite-dimensional compact metrizable topological space~$X$ whose fibers are stable Kirchberg algebras that satisfy the UCT and have rational $\K$\nb-theory groups. Assume that~$A$ and~$B$ have vanishing boundary maps. Then any isomorphism of $\K$\nb-theory cosheaves $\OK(A)\cong \OK(B)$ lifts to a $C(X)$-linear \Star{}isomorphism $A\cong B$.
\end{theorem}

The $\K$\nb-theory cosheaf~$\OK$ is a rather simple (but large) $\K$\nb-theoretic invariant which we shall define in~\S\ref{sec:vanishing_boundary_maps}; it comprises the $\K$\nb-theory groups of all (distinguished) ideals of the algebra, together with the maps induced by all inclusions of such ideals.

The proof of this theorem is concluded in~\S\ref{sec:proof}. We provide a version of the theorem for unital algebras in Theorem~\ref{thm:unital}. An Abelian group~$G$ is \emph{rational} if it is isomorphic to its tensor product with the field of rational numbers~$\Q$; this is equivalent to $G$ being torsion-free and divisible and to $G$ being a vector space over~$\Q$.

Our method of proof is largely parallel to the one in~\cite{BD}, where the UCT from~\cite{BK} for \Cstar{}algebras over finite accordion spaces was used, based on the observation that the unit inverval has sufficiently many finite quotients of accordion type. The main result of the present article is instead valid for an arbitrary finite-dimensional compact metrizable base space, but this comes at the expense of the assumption of vanishing boundary maps.

Using a result from~\cite{HRW}, Kirchberg's isomorphism theorem for $\Cuntz_2$\nb-absorbing \Cstar{}algebras mentioned above implies in particular that a separable continuous \Cstar{}bundle over a finite-dimensional compact metrizable space~$X$ whose fibers are stable UCT Kirchberg algebras with trivial $\K$\nb-theory groups is isomorphic to the trivial \Cstar{}bundle $C(X,\Cuntz_2\otimes\Compacts)$; see also~\cite{Dadarlat:Finite-dimensional}. Our classification result may be considered as an extension of this automatic triviality theorem for continuous $\Cuntz_2\otimes\Compacts$-bundles: instead of asking the \Cstar{}bundles to have entirely trivial $\K$\nb-theory, we only require the collection of the $\K$\nb-theory groups of all ideals in the algebra to form a flabby cosheaf of $\Z/2$\nb-graded $\Q$\nb-vector spaces (this terminology is explained in~\S\ref{sec:vanishing_boundary_maps}).

In Section~\ref{sec:range}, we determine the range of the invariant in the classification result above under the additional assumption of finite-dimensional $\K$\nb-theory. More precisely, we show:

\begin{theorem}
    \label{thm:range}
Let~$M$ be a flabby cosheaf of $\Z/2$\nb-graded $\Q$\nb-vec\-tor spaces on~$X$ such that $M(X)$ is finite-dimensional. Then~$M$ is a direct sum of a finite number of skyscraper cosheaves and $M\cong\OK(A)$ for a continuous field~$A$ as in Theorem~\textup{\ref{thm:main}}.
\end{theorem}

This theorem also has an analogue for unital continuous fields. The range question in the general case where~$M(X)$ may be countably infinite-dimensional remains open. We provide an example to illustrate the greater complexity in this case.

\subsection*{Acknowledgement}
The author is grateful to Marius Dadarlat for helpful conversations on the topic of the present article. He thanks the referee for a number of valuable suggestions.

\section{Preparations}

Throughout this article, we let~$X$ denote a finite-dimensional compact metrizable topological space (arbitrary topological spaces will be denoted by~$Y$). The topology of~$X$ (its lattice of open subsets) is denoted by~$\Open(X)$. We choose an ordered basis $(U_n)_{n\in\N}$ for~$\Open(X)$ and consider the (finite) $T_0$\nb-quotient~$X_n$ of~$X$ equipped with the topology~$\Open(X_n)$ generated by the family $\{U_1,\ldots,U_n\}$ (see \cite{DM}*{\S3}).

Our reference for continuous fields of \Cstar{}algebras (or, synonymically, \Cstar{}bun\-dles) is~\cite{Dixmier:Cstar}. For basic definitions, facts and notation concerning \Cstar{}algebras over (possibly non-Hausdorff) topological spaces, we refer to~\cite{MN:Bootstrap}. Versions of $\KK$\nb-theory and $\E$\nb-theory for separable \Cstar{}algebras over second countable topological spaces have been constructed in~\cite{MN:Bootstrap} and~\cite{DM}, respectively. By \cite{DM}*{Theorem 3.2}, there is a natural short exact sequence of $\Z/2$-graded Abelian groups
\begin{equation}
      \label{eq:approximation}
  \varprojlim\nolimits^1 \E_{*+1}(X_n;A,B)
  \into \E_*(X;A,B) \prto
  \varprojlim \E_*(X_n;A,B)
\end{equation}
for every pair $A$, $B$ of separable \Cstar{}algebras over~$X$.

Recall from \cite{MN:Bootstrap}*{\S3.2} that there is an exterior product for $\KK(X)$\nb-theory. In particular, we can form the (minimal) tensor product of a \Cstar{}algebra~$A$ over~$X$ with a \Cstar{}algebra~$D$ and obtain a \Cstar{}algebra $A\otimes D$ over~$X$. We let $M_\Q$ denote the universal UHF-algebra. Hence $\K_0(M_\Q)\cong\Q$ and $\K_1(M_\Q)=0$. A \Cstar{}algebra~$B$ (over~$X$) is called \emph{$M_\Q$\nb-absorbing} if $B\cong B\otimes M_\Q$. If~$A$ and~$B$ are \Cstar{}algebras over~$X$ and~$B$ is $M_\Q$\nb-absorbing, then the exterior product
\begin{equation}
  \label{eq:Q-vector_space_structure}
\KK_*(X;A,B)\otimes\KK_*(\C,M_\Q)\to\KK_*(X;A\otimes\C,B\otimes M_\Q) \cong\KK_*(X;A,B)
\end{equation}
turns $\KK_*(X;A,B)$ into a rational vector space.

\section{Vanishing boundary maps and flabby cosheaves}
  \label{sec:vanishing_boundary_maps}

A \Cstar{}algebra~$A$ over an arbitrary topological space~$Y$ is said to have \emph{vanishing boundary maps} if the natural map $i_U^V\colon\K_*\bigl(A(U)\bigr)\to\K_*\bigl(A(V)\bigr)$ is injective for all open subsets $U\subseteq V\subseteq Y$ (it suffices to consider the case $V=Y$); this is equivalent to the condition in \cite{Bentmann:Real_rank_zero_and_intermediate_cancellation}*{Definition~3.2} because of the six-term exact sequence.

If~$A$ has vanishing boundary maps, one can deduce from the Mayer--Vietoris sequence and continuity of $\K$\nb-theory that, for every covering $(V_i)_{i\in I}$ of an open subset $V\subseteq Y$ by open subsets $V_i\subseteq V$, one has an exact sequence
\begin{equation}
  \label{eq:cosheaf_condition}
\bigoplus_{j,k\in I} \K_*\bigl(A(V_j\cap V_k)\bigr)
\xrightarrow{(i_{V_j\cap V_k}^{V_j}-i_{V_j\cap V_k}^{V_k})}
\bigoplus_{i\in I} \K_*\bigl(A(V_i)\bigr)
\xrightarrow{(i_{V_i}^U)}
\K_*\bigl(A(V)\bigr) \longrightarrow 0;
\end{equation}
see \cite{Bredon:Cosheaves}*{Proposition~1.3}.

\begin{definition}
 \label{def:K-theory_cosheaf}
The \emph{$\K$\nb-theory cosheaf} $\OK^Y(A)$ of a \Cstar{}algebra~$A$ over~$Y$ with vanishing boundary maps consists of the collection of $\Z/2$\nb-graded Abelian groups $\bigl(\K_*\bigl(A(U)\bigr)\mid U\in\Open(Y)\bigr)$ together with the collection $\bigl(i_U^V\mid V\in\Open(Y), U\in\Open(V)\bigr)$ of graded group homomorphisms.
\end{definition}

In Theorem~\ref{thm:main} we briefly wrote $\OK(A)$ for $\OK^X(A)$. For an arbitrary \Cstar{}al\-ge\-bra~$A$ over~$Y$, $\OK^Y(A)$ would only define a precosheaf, that is, a covariant functor on~$\Open(X)$. However, if~$A$ has vanishing boundary maps, then by~\eqref{eq:cosheaf_condition}, $\OK^Y(A)$ is indeed a \emph{flabby cosheaf} of $\Z/2$\nb-graded Abelian groups in the technical sense of~\cite{Bredon:Cosheaves}*{\S1}. We reproduce the definition below for the reader's convenience. The partially ordered set~$\Open(Y)$ is considered as a category with morphisms given by inclusions.

\begin{definition}
A \emph{precosheaf} on~$Y$ is a covariant functor~$M$ from $\Open(Y)$ to the category of modules over some ring. For open subsets $U\subseteq V\subseteq Y$, the induced map $M(U)\to M(V)$ is denoted by~$i_U^V$. A precosheaf~$M$ is a \emph{cosheaf} if the sequence
\begin{equation}
  \label{eq:general_cosheaf_exact_sequence}
\bigoplus_{j,k\in I} M(V_j\cap V_k)
\xrightarrow{(i_{V_j\cap V_k}^{V_j}-i_{V_j\cap V_k}^{V_k})}
\bigoplus_{i\in I} M(V_i)
\xrightarrow{(i_{V_i}^U)}
M(V)\longrightarrow 0
\end{equation}
is exact for every open covering $(V_i)_{i\in I}$ of an open subset $V\subseteq Y$. It is \emph{flabby} if the map $i_U^V\colon M(U)\to M(V)$ is injective for all open subset $U\subseteq V\subseteq Y$. A morphism of cosheaves is a natural transformation of the corresponding functors.
\end{definition}

\begin{remark}
  \label{remark:restrict_to_basis}
The invariant~$\OK$ is necessarily large (in a non-technical sense), as its purpose is to classify certain \Cstar{}bundles that need not be locally trivial. However, as we shall see in the proof of Theorem~\protect\ref{thm:main}, we may minimize its size without losing essential information by restricting to a fixed countable basis of~$\Open(X)$.
\end{remark}

\section{Proof of Theorem~\protect\ref{thm:main}}
  \label{sec:proof}
We are now prepared to prove our main result. As in Theorem~\ref{thm:main}, we abbreviate $\OK^X$ by $\OK$. Assume that~$A$ and~$B$ as in Theorem~\ref{thm:main} are given. By \cite{BD}*{Propositions~2.8 and~2.10}, $A$~and~$B$ are stable, nuclear and $\Cuntz_\infty$\nb-ab\-sor\-bing, and belong to the $\E(X)$\nb-theoretic bootstrap class $\Bootstrap_\E(X)$ defined in \cite{DM}*{Definition~4.1}. Hence, by the comparison theorem \cite{DM}*{Theorem~5.4}, Kirchberg's classification theorem~\cite{Kirchberg}*{Folgerung~4.3} and the invertibility criterion~\cite{DM}*{Theorem~4.6}, it suffices to show that a given homomorphism $\OK(A)\to \OK(B)$ lifts to an element in $\E_0(X;A,B)$.

Consider now one of the approximating spaces~$X_n$ (this may have the homeomorphism type of any finite $T_0$\nb-space). We apply the machinery of homological algebra in triangulated categories~\cite{MN:HomologicalI} to the category $\Ecat(X_n)\otimes\Q$ whose objects are separable \Cstar{}algebras over~$X_n$ and whose morphism groups are given by $\E_0(X_n,\blank,\blank)\otimes\Q$. This is indeed a triangulated category because $\Ecat(X_n)$ is triangulated (see~\cite{DM}) and $\Q$ is flat (compare~\cite{Inassaridze-Kandelaki-Meyer}). In~\cite{Bentmann:Real_rank_zero_and_intermediate_cancellation}*{\S4--5}, we defined a homological functor $\XnK$ on $\KKcat(X_n)$ and indicated that it is universal with respect to the homological ideal given by the kernel of $\XnK$ on morphisms. Similarly, we have a homological functor $\XnK\otimes\Q$ on $\Ecat(X_n)\otimes\Q$ that is universal with respect to its kernel. Since we are working over the field~$\Q$, an argument as in \cite{Bentmann:Real_rank_zero_and_intermediate_cancellation}*{Lemma~4.6} shows that $\XnK(A)$ is projective since $A$ has vanishing boundary maps. From \cite{MN:HomologicalI}*{Theorem~3.41}, we obtain an isomorphism $\E_0(X_n;A,B)\otimes\Q\cong\Hom\bigl(\XnK(A)\otimes\Q,\XnK(B)\otimes\Q\bigr)$. By \cite{Bentmann:Real_rank_zero_and_intermediate_cancellation}*{Lemma~4.3} and the natural isomorphisms $\Colim\circ\XnK\cong\OK^{X_n}$ and $\Res\circ\OK^{X_n}\cong\XnK$, we may replace $\XnK$ with $\OK^{X_n}$ and obtain the identification
\begin{equation}
   \label{eq:finite_rational_UCT}
\E_0(X_n;A,B)\otimes\Q\cong\Hom\bigl(\OK^{X_n}(A)\otimes\Q,\OK^{X_n}(B)\otimes\Q\bigr). 
\end{equation}
Since the fibers of~$A$ and~$B$ have rational $\K$\nb-theory groups, they are $M_\Q$\nb-absorbing by the Kirchberg--Phillips classification theorem~\cite{Rordam:Classification_survey}*{\S8.4}. By~\cite{HRW}, the algebras~$A$ and~$B$ themselves are $M_\Q$\nb-absorbing. Hence, by~\eqref{eq:Q-vector_space_structure}, \eqref{eq:finite_rational_UCT}, the comparison theorem \cite{DM}*{Theorem~5.4} and the K\"unneth formula for tensor products,
\begin{equation}
  \label{eq:uct_for_X_n}
\E_0(X_n;A,B)\cong\Hom\bigl(\OK^{X_n}(A),\OK^{X_n}(B)\bigr)
\end{equation}
because $\Q\otimes\Q\cong\Q$. The continuity of $\K$\nb-theory with respect to inductive limits shows that the canonical map $\Hom\bigl(\OK(A),\OK(B)\bigr)\to\varprojlim\Hom\bigl(\OK^{X_n}(A),\OK^{X_n}(B)\bigr)$ is an isomorphism. Hence $\varprojlim \E_0(X_n;A,B)\cong\Hom\bigl(\OK(A),\OK(B)\bigr)$. The claim now follows from the exact sequence~\eqref{eq:approximation}.

\subsection{Classification of unital \Cstar{}bundles}
For a unital \Cstar{}bundle over~$X$, we may equip $\OK(A)$ with the unit class $[1_A]\in\K_0(A)$. This pair is denoted by $\OK^+(A)$; it is a \emph{pointed} cosheaf, that is, a cosheaf~$M$ with a distinguished element in the degree-zero part of~$M(X)$. Morphisms of such pointed cosheaves are of course required to preserve the distinguished element. By \cite{Eilers-Restorff-Ruiz:generalized_meta_theorem}*{Theorem~3.3}, we immediately obtain the following version of our main result for unital algebras.

\begin{theorem}
  \label{thm:unital}
Let $A$ and $B$ be separable unital continuous fields over~$X$ whose fibers are UCT Kirchberg algebras with rational $\K$\nb-theory groups. Assume that $A$~and~$B$ have vanishing boundary maps. Then any isomorphism $\OK^+(A)\cong \OK^+(B)$ lifts to a $C(X)$-linear \Star{}isomorphism $A\cong B$.
\end{theorem}

\section{Range results}
  \label{sec:range}
We investigate the question of the range of the invariant in Theorem~\ref{thm:unital} (the same considerations apply mutatis mutandis and without keeping track of the unit class in the stable case and yield a proof for Theorem~\ref{thm:range}). For the results in this section, it would suffice to assume that~$X$ is a compact metrizable space.

If a unital \Cstar{}bun\-dle~$A$ of the form classified by our result is locally trivial on an open subset~$U$ of~$X$, then $A(U)$ must be isomorphic to $C_0(U,\Cuntz_2)$. Hence interesting examples cannot be locally trivial (around every point in~$X$).

\begin{example}
We will now describe some basic non-trivial examples of \Cstar{}bundles satisfying the conditions in our classification theorem. Let~$D_1$, $\ldots$, $D_n$ be unital UCT Kirchberg algebras. By the Exact Embedding Theorem \cite{Kirchberg-Phillips:Embedding}*{Theorem~2.8}, we may find unital \Star{}monomorphisms $\gamma_i\colon D_i\to\Cuntz_2$. For points~$x_1$, $\ldots$, $x_n$ in~$X$, we define
\begin{equation}
  \label{eq:field_with_finitely_many_singularities}
 A=\{f\in C(X,\Cuntz_2)\mid \textup{$f(x_i)\in\gamma_i(D_i)$ for $i=1,\ldots,n$}\}.
\end{equation}
This is clearly a continuous field of Kirchberg algebras, with fiber~$D_i$ at~$x_i$ and fiber~$\Cuntz_2$ at all other points. A simple computation using excision shows that
\[
 \K_*\bigl(A(U)\bigr)\cong\bigoplus_{i\colon x_i\in U}\K_*(D_i).
\]
Hence~$\OK(A)$ is the direct sum of so-called skyscraper cosheaves~$i_{x_i}\bigl(\K_*(D_i)\bigr)$ based at~$x_i$ with coefficient group~$\K_*(D_i)$. Here~$i_x(G)$ is defined by
\[
 i_x(G)(U)= \begin{cases} G & \textup{if $x\in U$,} \\ 0 & \textup{else}. \end{cases}
\]
These cosheaves are indeed flabby. It follows that the continuous field~$A$ has vanishing boundary maps. So, if the algebras~$D_i$ have rational $\K$\nb-theory groups, then~$A$ satisfies the conditions of Theorem~\ref{thm:unital}. Under the identification $\K_0(A)\cong\bigoplus_{i=1}^n\K_0(D_i)$, we have $[1_A]=\sum_{i=1}^n [1_{D_i}]$. Using the range result for $\K$\nb-theory on unital UCT Kirchberg algebras~\cite{Rordam:Classification_survey}*{\S4.3}, it follows that an arbitrarily pointed finite direct sum of skyscraper cosheaves whose coefficient groups are countable $\Z/2$\nb-graded Abelian groups can be realized as the pointed $\K$\nb-theory cosheaf of a unital continuous field as in~\eqref{eq:field_with_finitely_many_singularities}.
\end{example}

The following proposition shows that, if~$A$ is a unital continuous field as in Theorem~\ref{thm:unital} and the $\Q$\nb-vec\-tor space $\K_*(A)$ is finite-dimensional, then $A$ must be of the form~\eqref{eq:field_with_finitely_many_singularities}.

\begin{proposition}
  \label{pro:finite-dim_flabby}
Let~$\Field$ be a field and let~$Y$ be an arbitrary topological space. Let~$M$ be a flabby cosheaf of $\Field$\nb-vec\-tor spaces over~$Y$. If~$M(Y)$ is finite-dimensional, then $M$ is a direct sum of a finite number of skyscraper cosheaves.
\end{proposition}

\begin{proof}
We proceed by induction on the dimension of~$M(Y)$. If the dimension is zero, then there is nothing to prove. Otherwise, by \cite{Bredon:Sheaf_theory}*{V.~Propo\-si\-tion~1.5}, there exists $y\in Y$ such that $M(Y\setminus\overline{\{y\}})$ is a proper subspace of~$M(Y)$. By assumption, the subcosheaf~$N$ of~$M$ defined by
\[
N(U)=M(U\setminus\overline{\{y\}})
\]
for $U\in\Open(Y)$ is a direct sum of skyscraper cosheaves. Since the quotient $Q=M/N$ vanishes on $Y\setminus\overline{\{y\}}$, it follows from the exact sequence \eqref{eq:general_cosheaf_exact_sequence} that $Q$ is a skyscraper cosheaf of the form $Q=i_{y}(V)$ for some $\Field$\nb-vec\-tor space~$V$. It remains to show that the extension $N\rightarrowtail M\twoheadrightarrow Q$ splits. We have $\Hom(i_y(V),N)\cong\Hom\bigl(V,\varprojlim\limits_{U\ni x} N(U)\bigr)$ and thus
\[
\Ext^1(Q,N)\cong\Hom\bigl(V,\mathop{\varprojlim\nolimits^1}\limits_{U\ni x} N(U)\bigr)=0
\]
by the Mittag--Leffler condition using that $N(U)$ is a finite-dimensional vector space for every~$U$.
\end{proof}

The considerations above are summarized in the following version of Theorem~\ref{thm:range} for unital continuous fields:

\begin{theorem}
    \label{thm:unital_range}
Let~$(M,m)$ be a pointed flabby cosheaf of $\Z/2$\nb-graded $\Q$\nb-vec\-tor spaces on~$X$ such that $M(X)$ is finite-dimensional. Then~$M$ is a direct sum of a finite number of skyscraper cosheaves and $(M,m)\cong\OK^+(A)$ for a continuous field~$A$ as in Theorem~\textup{\ref{thm:unital}}.
\end{theorem}

Combining the range result above with our classification results, we obtain an explicit description of the isomorphism classes of the classified continuous fields~$A$ whose $\K$\nb-theory $\K_*(A)$ is finite-dimensional over~$\Q$. In the case that $\K_*(A)$ is an arbitrary (countable) $\Q$\nb-vec\-tor space the situation is unclear: it remains open whether a countable direct sum of skyscraper cosheaves whose coefficient groups are countable $\Z/2$\nb-graded $\Q$-vec\-tor spaces can be realized as the $\K$\nb-theory cosheaf $\OK(A)$ of a continuous field~$A$ as in Theorem~\ref{thm:main}.

We generalize the previous example by replacing the finite set of singularities with a totally disconnected subset. This demonstrates that the assumption of finite-dimensionality in Proposition~\ref{pro:finite-dim_flabby} cannot be dropped, that is, a flabby cosheaf of vector spaces of countable dimension need not be a direct sum of skyscraper cosheaves.

\begin{example}
  \label{exa:totally_disconnected}
Let $Y\subseteq X$ be a closed, totally disconnected subset. Then $C(Y)$ is a tight \Cstar{}algebra over $Y$ with vanishing boundary maps. This follows because the \Cstar{}algebra $C(Y)$ as well as all of its ideals and quotients are AF\nb-algebras and thus have vanishing $\K_1$\nb-groups. In fact, one readily sees that $\OK\bigl(C(Y)\bigr)\cong C_c(\blank,\Z)$ (compactly supported, locally constant functions).

We let $A$ be the tight \Cstar{}algebra over $Y$ given by $C(Y)\otimes\Cuntz_\infty\otimes\UHF$. This is a unital continuous field of Kirchberg algebras over~$Y$. By Blanchard's embedding theorem~\cite{Blanchard:Subtriviality}, there is a unital \Star{}monomorphism $\alpha\colon A\hookrightarrow C(Y)\otimes\Cuntz_2$ over~$Y$. We define
\[
B=\{f\in C(X)\otimes\Cuntz_2\mid f|_Y\in\alpha(A)\}.
\]
Excision by the obvious extension $C_0(X\setminus Y)\otimes\Cuntz_2\rightarrowtail B\twoheadrightarrow A$ together with the K\"unneth formula shows that $B$ has vanishing boundary maps. By construction, $B$ is a unital separable continuous field over~$X$ whose fibers are UCT Kirchberg algebras with rational $\K$\nb-theory groups; more precisely, we have $B(x)\cong\Cuntz_\infty\otimes\UHF$ for $x\in Y$ and $B(x)\cong\Cuntz_2$ for $x\not\in Y$. Hence the restriction of $B$ to any closed subset that intersects both $Y$ and $X\setminus Y$ is nontrivial. The $\K$\nb-theory cosheaf of $B$ is given explicitly by $$U\mapsto C_c(U\cap Y,\Z)\otimes\Q.$$
This is a flabby cosheaf of rational vector spaces of countable dimension. The stalks of $\OK(B)$ are given by $\Q$ at points in $Y$ and by $0$ at points in $X\setminus Y$. (The stalk $M_x$ of a cosheaf $M$ on $X$ at a point $x\in X$ is the quotient $M(X)/M(X\setminus\{x\})$.) I~know no separable continuous field~$C$ over~$X$ with vanishing boundary maps such that the set $\{x\in X\mid \OK(C)_x\neq 0\}$ is not zero-dimensional.
\end{example}

\section{Further remarks}
  \label{sec:further_remarks}

\subsection{Real rank zero}
We briefly comment on the relationship of the assumptions in our classification theorem to real rank zero, a property that has often proved useful for classification purposes. It was shown in \cite{Pasnicu_Rordam:PI}*{Theorem~4.2} that a separable purely infinite \Cstar{}algebra~$A$ has real rank zero if and only if the primitive ideal space of~$A$ has a basis consisting of compact open subsets and~$A$ is $\K_0$\nb-liftable (meaning, in our terminology, that $A$~has ``vanishing exponential maps''). While a \Cstar{}bundle (with non-vanishing fibers) over a compact metrizable space of positive dimension cannot satisfy the first condition, the second condition of $\K_0$\nb-liftability is built into our assumptions (we also assume that $A$~has ``vanishing index maps''). As Theorem~\ref{thm:range} shows, at least in the case of finite-dimensional $\K$\nb-theory, the $\K$\nb-theory cosheaf of a separable continuous field with vanishing boundary maps has a very zero-dimensional flavour.

\subsection{Cosheaves versus sheaves}
The following explanations clarify the relationship (in the setting of fields with vanishing boundary maps) between our $\K$\nb-theory cosheaf and the $\K$\nb-theory sheaf defined in~\cite{DE} for \Cstar{}bundles over the unit interval. In \cite{Bredon:Sheaf_theory}*{Propositions~V.1.6 and~V.1.8}, Glen Bredon provides a structure result for flabby cosheaves: the compact sections functor yields a one-to-one correspondence between soft sheaves and flabby cosheaves on~$\Open(X)$. A sheaf is \emph{soft} if sections over closed subsets can be extended to global sections. If $\widehat{\OK(A)}$ denotes the soft sheaf corresponding to the flabby cosheaf $\OK(A)$, then we have $\widehat{\OK(A)}(Z)\cong\K_*\bigl(A(Z)\bigr)$ for every \emph{closed} subset $Z\subseteq X$. Regarding the range question considered in~\S\ref{sec:range}, we remark that \cite{Dadarlat--Elliott--Niu}*{The\-o\-rem~5.8} provides a range result for unital \Cstar{}bundles over the unit interval, but it is not clear when the constructed algebras have vanishing boundary maps.

\subsection{Another classification result}
We conclude the note by stating one more result which follows in essentially the same way as our main result. We comment below on the required modifications in the proof. Again, a version for unital algebras can be obtained from \cite{Eilers-Restorff-Ruiz:generalized_meta_theorem}*{Theorem~3.3}.

\begin{theorem}
  \label{theorem:parity}
Fix $i\in\{0,1\}$. Let~$A$ and~$B$ be separable continuous fields of stable UCT Kirchberg algebras over a finite-dimensional compact metrizable topological space~$X$. Assume that $\K_i\bigl(A(Z)\bigr)=0$ for all locally closed subsets $Z\subseteq X$. Then any isomorphism $\OK(A)\cong \OK(B)$ lifts to a $C(X)$-linear \Star{}isomorphism $A\cong B$.
\end{theorem}

Notice that we do not assume that the fibers of~$A$ and~$B$ have rational $\K$\nb-theory groups. Example~\ref{exa:totally_disconnected} provides an example of a non-trivial \Cstar{}bundle falling under the classification in Theorem~\ref{theorem:parity}.

The $\K$\nb-the\-o\-ret\-i\-cal assumption in the theorem implies that~$A$ and~$B$ have vanishing boundary maps. Hence the universal coefficient theorem \cite{Bentmann:Real_rank_zero_and_intermediate_cancellation}*{Theorem~5.2} applies (we may write $\OK^{X_n}$ instead of $\XnK$ by \cite{Bentmann:Real_rank_zero_and_intermediate_cancellation}*{Lemma~4.3}) and simplifies to an isomorphism because the relevant $\Ext^1$-term vanishes for parity reasons. The remainder of the proof is analogous.

\end{document}